\documentclass[11pt]{amsart}
\usepackage{geometry}                
\usepackage{graphicx}
\usepackage{amssymb}
\usepackage{epstopdf}
\newtheorem{theorem}{Theorem}[section]
\newtheorem{lemma}{Lemma}[section]
\newtheorem{corollary}{Corollary}[section]

\newtheorem{definition}{Definition}[section]

\newtheorem{remark}{Remark}[section]
\newtheorem{question}{Question}[section]
\numberwithin{equation}{section}

\def\Z{\Bbb Z}

\def\R{\Bbb R}

\def\d{\partial}

\def\e{\epsilon}
\def\a{\alpha}
\def\b{\beta}
\def\g{\gamma}

\title{The Disk-Based Origami Theorem and a Glimpse of Holography for Traversing Flows}
\author{Gabriel Katz}
\address{MIT, Department of Mathematics, 77 Massachusetts Ave., Cambridge, MA 02139, U.S.A.}
\email{gabkatz@gmail.com}

\begin{document}

\maketitle

\begin{abstract} 
This paper describes a mechanism by which a traversally generic flow $v$ on a smooth connected manifold $X$ with boundary produces a compact $CW$-complex $\mathcal T(v)$, which is homotopy equivalent to $X$ and such that $X$ embeds in $\mathcal T(v)\times \R$. The $CW$-complex $\mathcal T(v)$ captures some residual information about the smooth structure on $X$ (such as the stable tangent bundle of $X$). Moreover, $\mathcal T(v)$ is obtained from a simplicial \emph{origami map} $O: D^n \to \mathcal T(v)$, whose source space is a disk $D^n \subset \d X$ of dimension $n = \dim(X) -1$. The fibers of $O$ have the cardinality $(n+1)$ at most. 

The knowledge of the map $O$, together with the restriction to $D^n$ of a Lyapunov function $f:X \to \R$ for $v$, make it possible to reconstruct the topological type of the pair $(X, \mathcal F(v))$, were $\mathcal F(v)$ is the $1$-foliation, generated by $v$. This fact motivates the use of ``holography" in the title. 
\end{abstract}


\section{Trivia about traversing flows on manifolds with boundary}

The results of this paper are quite direct implications of our study of boundary generic and traversally generic flows in \cite{K1}, \cite{K2}. \smallskip

For the reader convenience, we start with presenting few basic definitions and facts related to the \emph{boundary generic traversing} and \emph{traversally generic} vector fields on manifolds with boundary.

Let $X$ be a compact connected smooth $(n+1)$-dimensional manifold with boundary. A vector field $v$ is called \emph{traversing} if each $v$-trajectory is ether a closed interval with both ends residing in $\d X$, or a singleton also residing in $\d X$ (see \cite{K1} for the details). In fact, $v$ is traversing if and only if it admits a smooth Lyapunov function $f: X \to \R$, such that $df(v) > 0$ in $X$ (see \cite{K1}). 

For traversing fields $v$, the trajectory space $\mathcal T(v)$ is homology equivalent  to $X$ (Theorem 5.1, \cite{K3}).   \smallskip

We denote by $\mathcal V_{\mathsf{trav}}(X)$ the space of traversing fields on $X$. \smallskip

We consider an important subclass of traversing fields which we call \emph{traversally generic} (see formula (2.4) and Definition 3.2 from \cite{K2}).

For a traversally  generic field $v$, the trajectory space $\mathcal T(v)$ is stratified by closed subspaces, labeled by the elements $\omega$ of an \emph{universal} poset $\Omega^\bullet_{'\langle n]}$, which depends only on $\dim(X) = n+1$ (see \cite{K3}, Section 2, for the definition and properties of $\Omega^\bullet_{'\langle n]}$). The elements $\omega \in \Omega^\bullet_{'\langle n]}$ correspond to combinatorial patterns that describe the way in which $v$-trajectories $\g \subset X$ intersect the boundary $\d X$. Each intersection point $a \in \g \cap \d X$ acquires a well-defined \emph{multiplicity} $m(a)$, a natural number that reflects the order of tangency of $\g$ to $\d X$ at $a$ (see \cite{K1} and Definition 2.1 for the expanded definition of $m(a)$). So $\g \cap \d_1X$ can be viewed as a \emph{divisor} $D_\g$ on $\g$, an ordered set of points in $\g$ with their multiplicities. Then $\omega$ is just the ordered sequence of multiplicities $\{m(a)\}_{a \in \g \cap \d X }$, the order being prescribed by $v$. 

The support of the divisor $D_\g$ is: either (1) a singleton $a$, in which case $m(a) \equiv 0 \; \mod \, 2$, or (2) the minimum and maximum points of $\sup D_\g$ have \emph{odd} multiplicities, and the rest of the points have \emph{even} multiplicities. \smallskip

Let 
\begin{eqnarray}\label{eq2.1}
m(\g) =_{\mathsf{def}} \sum_{a \in \g\, \cap \, \d_1X }\; m(a)\quad \text{and} \quad m'(\g) =_{\mathsf{def}} \sum_{a \in \g \, \cap \, \d_1X }\; (m(a) -1).
\end{eqnarray} 
Similarly, for $\omega =_{\mathsf{def}} (\omega_1, \omega_2, \dots , \omega_i,  \dots )$ we introduce the \emph{norm} and the \emph{reduced norm} of $\omega$ by the formulas: 

\begin{eqnarray}\label{eq2.2}
|\omega| =_{\mathsf{def}} \sum_i\; \omega_i \quad \text{and} \quad |\omega|'  =_{\mathsf{def}} \sum_i\; (\omega_i -1).
\end{eqnarray}
\smallskip

We assume that $X$ is embedded in a larger smooth manifold $\hat X$, and the vector field $v$ is extended to a non-vanishing vector field $\hat v$ in $\hat X$. We treat the pair $(\hat X, \hat v)$ as a \emph{germ}, containing $(X, v)$.

Let $\d_jX =_{\mathsf{def}} \d_jX(v)$ denote the locus of points $a \in \d_1X$ such that the multiplicity of the $v$-trajectory $\g_a$ through $a$ at $a$ is greater than or equal to $j$. (By definition, $\d_1X(v) = \d X$.) This locus has a description in terms of  an auxiliary function $z: \hat X \to \R$ which satisfies the following three properties:
\begin{eqnarray}\label{eq2.3}
\end{eqnarray}

\begin{itemize}
\item $0$ is a regular value of $z$,   
\item $z^{-1}(0) = \d_1X$, and 
\item $z^{-1}((-\infty, 0]) = X$. 
\end{itemize}

In terms of $z$, the locus $\d_jX =_{\mathsf{def}} \d_jX(v)$ is defined by the equations: 
$$\{z =0,\; \mathcal L_vz = 0,\; \dots, \;  \mathcal L_v^{(j-1)}z = 0\},$$
where $\mathcal L_v^{(k)}$ stands for the $k$-th iteration of the Lie derivative operator $\mathcal L_v$ in the direction of $v$ (see \cite{K2}). 

The pure stratum $\d_jX^\circ \subset \d_jX$ is defined by the additional constraint  $\mathcal L_v^{(j)}z \neq 0$. The locus $\d_jX$ is the union of two loci: $(1)$ $\d_j^+X$, defined by the constraint  $\mathcal L_v^{(j)}z \geq  0$, and $(2)$ $\d_j^-X$, defined by the constraint  $\mathcal L_v^{(j)}z \leq  0$. The two loci, $\d_j^+X$ and $\d_j^-X$, share a common boundary $\d_{j+1}X$.
\smallskip

\noindent{\bf Definition 2.1} The multiplicity $m(a)$, where $a \in \d X$, is the index $j$ such that $a \in \d_jX^\circ$. 

\hfill $\diamondsuit$
\smallskip

The characteristic property of \emph{traversally generic} fields is that they admit special flow-adjusted coordinate systems, in which the boundary is given by  quite special polynomial equations (see formula (\ref{eq2.4})) and the trajectories are parallel to one of the preferred coordinate axis (see  \cite{K2}, Lemma 3.4). For a traversally generic $v$ on a $(n+1)$-dimensional $X$, the vicinity $U \subset \hat X$ of each $v$-trajectory $\g$ of the combinatorial type $\omega$ has a special coordinate system $$(u, \vec x, \vec y): U \to \R\times \R^{|\omega|'} \times \R^{n-|\omega|'}.$$ By Lemma 3.4  and formula $(3.17)$ from \cite{K2}, in these coordinates, the boundary $\d_1X$ is given  by the polynomial equation: 
\begin{eqnarray}\label{eq2.4}
 \wp (u, \vec x) =_{\mathsf{def}} \prod_i \big[(u-i)^{\omega_i} + \sum_{l = 0}^{\omega_i-2} x_{i, l}(u -i)^l \big] = 0
 \end{eqnarray}
of an even degree $|\omega|$ in $u$. Here  $i \in \Z$ runs over the distinct roots of  $\wp (u, \vec 0)$ and  $\vec x =_{\mathsf{def}} \{ x_{i, l}\}_{i,l}$.   
At the same time, $X$ is given by the polynomial inequality $\{\wp(u, \vec x) \leq 0\}$.  Each $v$-trajectory in $U$ is produced by freezing all the coordinates $\vec x, \vec y$, while letting $u$ to be free.
\smallskip

We denote by $\mathcal V^\ddagger(X)$ the space of traversally  generic fields on $X$.  It turns out that $\mathcal V^\ddagger(X)$ is an \emph{open} and \emph{dense} (in the $C^\infty$-topology) subspace of $\mathcal V_{\mathsf{trav}}(X)$ (see \cite{K2}, Theorem 3.5).
\smallskip

We denote by $X(v, \omega)$ the union of $v$-trajectories whose divisors are of a given combinatorial type $\omega \in \Omega^\bullet_{'\langle n]}$. Its closure $\cup_{\omega' \preceq_\bullet \omega}\; X(v, \omega')$ is denoted by $X(v, \omega_{\succeq_\bullet})$.
\smallskip

Each pure stratum $\mathcal T(v, \omega) \subset \mathcal T(v)$ is an open smooth manifold and, as such, has a ``conventional" tangent bundle. \smallskip


\section{How traversally generic flows generate the origami homotopy models of manifolds with boundary}
Let $X$ be an $(n + 1)$-dimensional compact connected smooth manifold, carrying a traversally  generic vector field $v$. Abusing notations, we use the same symbol ``$\g$" for the $v$-trajectory in $X$ and for the point in the trajectory space $\mathcal T(v)$ it represents.\smallskip

We introduce a new filtration $\big\{\mathcal T_{\{max \geq k\}}^+(v)\big\}_{k\, \in\, [1, n+1]}$ of the trajectory space $\mathcal T(v)$ by closed subspaces (actually, by cellular subcomplexes). This stratification is cruder than the stratification $\big\{\mathcal T(v, \omega_{\succeq_\bullet})\big\}_{\omega\, \in\, \Omega^\bullet_{'\langle n]}}$. By definition, a trajectory $\g \in \mathcal T_{\{max \geq k\}}^+(v)$ if $\g \cap \d_1^+X$ contains at least one point $x$ of multiplicity greater than or equal to $k$. Moreover, we insist that  this  $x \in \d_k^+X$ (and not in $\d_k^-X^\circ$). In other words, $ \mathcal T_{\{max \geq k\}}^+(v)$ is exactly the image of $\d_k^+X(v)$ in the trajectory space under the obvious map $\Gamma: X \to \mathcal T(v)$.  In particular, $\mathcal T_{\{max \geq 1\}}^+(v) = \mathcal T(v)$. 
\smallskip

We notice that Corollary 3.3 from \cite{K1} and Theorems 3.4 and 3.5 from \cite{K2} imply that there is an nonempty open subset $\mathcal D(X) \subset \mathcal V^\ddagger(X)$  such that, for each field  $v \in \mathcal  D(X)$, all the strata $\{\d_j^+X\}_j$ are diffeomorphic to closed balls, except for $\d_n^+X$ (which  is a finite union of 1-balls) and for the finite set $\d_{n+1}^+X$. 

So let us consider a model filtration
\begin{eqnarray}\label{eq6.47} 
Z^0 \subset Z^1 \subset D^2 \subset D^3 \subset \dots D^{n-1} \subset D^n
\end{eqnarray}
of  a closed ball $D^n$, such that: 

(1) each ball $D^j \subset \d D^{j+1}$, 

(2)  $Z^1$ is a disjoint union of finitely many arcs in $\d D^2$, 

(3) $Z^0 \subset \d Z^1$ is a finite set. 
\smallskip

\begin{figure}[ht]\label{fig1}
\centerline{\includegraphics[height=2.5in,width=3.5in]{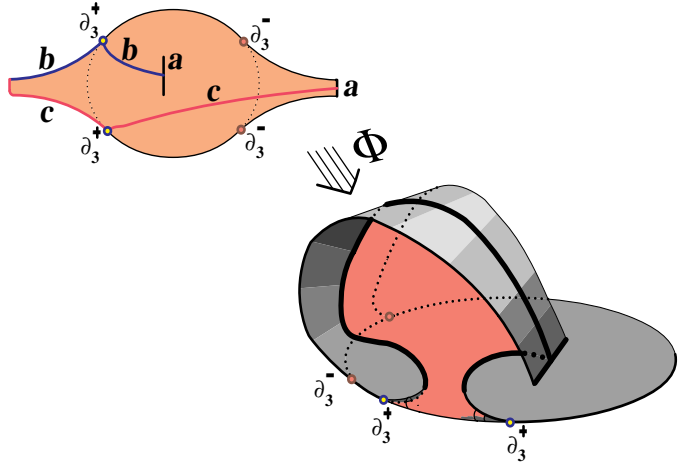}}
\bigskip
\caption{\small{An origami map $O: D^2 \rightarrow K$ of the $2$-disk onto a collapsable $2$-complex $K$. Note the pairs of arcs in $D^2$, marked by the labels $a, b, c$; each pair is identified by $O$ into a single arc, residing in $K$. The map  $O$ is a $3$-to-$1$ at most.}}
\end{figure}

It turns out that, for $v \in \mathcal  D(X)$, the trajectory space $\mathcal T(v)$ can be produced by an origami-like folding of the ball $D^n$ (see Figure \ref{fig1} for an example of an origami map on a $2$-ball). \smallskip

The result below should be compared with Theorem 1 from \cite{FR}.  It claims that a closed $3$-manifold $X$ has a spine which is the image of an immersed  $2$-sphere in general position in $X$. Theorem \ref{th6.11} should be compared also with somewhat  similar Theorem 5.2 in \cite{K}, the latter dealing with the flow-generated spines, not trajectory spaces (see \cite{GR} for the brief description of spines). 

\begin{theorem}\label{th6.11} {\bf (Trajectory spaces as the ball-based origami)} \hfill\break
Any compact connected smooth $(n+1)$-manifold $X$ with boundary admits a traversally  generic vector field $v$ such that:

\begin{itemize}
\item  its trajectory space $\mathcal T(v)$ is the image of a closed ball $D^n \subset \d_1X$ under a continuous  cellular  map $\Gamma: D^n \to \mathcal T(v)$, which is $(n +1)$-to-$1$ at most. \smallskip

\item The $\Gamma$-image in $\mathcal T(v)$ of each ball $D^k$ from the filtration $(\ref{eq6.47})$, is the space \hfill\break $\mathcal T_{\{max \geq n + 1- k\}}^+(v)$, and the restriction 
$\Gamma|_{D^k}$ is a $\lceil \frac{n}{n - k} \rceil$-to-$1$ map at most.  

For $n > 2$, the maps $$\Gamma: Z^1 \to \mathcal T_{\{max \geq n\}}^+(v)\; \; \text{and} \;\;  \Gamma: Z^0 \to \mathcal T_{\{max \geq n +1\}}^+(v)$$ are both bijective. \smallskip

\item The restrictions of $\Gamma$ to $\d D^{k + 1} \setminus D^k$ are $1$-to-$1$ maps for all  $1 < k < n$, and so are the restrictions of $\Gamma$ to $\d D^2 \setminus Z^1$ and $\d Z^1 \setminus Z^0$. \smallskip

\item The vector fields $v$, for which the above properties hold, form an open nonempty set in the space $\mathcal V^\ddagger(X)$ of traversally  generic fields on $X$, and thus an open set in the space $\mathcal V_{\mathsf{trav}}(X)$ of all traversing fields. 
\end{itemize}
\end{theorem}

\begin{proof}  If $\d_1X$ has several connected components, we pick one of them, say, $\d_1^\star X$. The union of the remaining boundary components  is denoted by $\d_1^{\star\star}X$. We can construct a Morse function $f: X \to \R$ so that it is locally constant on $\d_1^{\star\star}X$, and these constants are the local maxima of $f$ in a collar $U$ of $\d_1^{\star\star}X$ in $X$. Then by finger moves (as in the proof of Lemma 3.2, \cite{K1}) we eliminate all critical points of $f$ without changing $f$ in $U$. Pick a Riemmanian metric on $X$ and let $v$ be the gradient field of $f$.
Evidently, $\d_1^{\star\star}X \subset \d_1^-X(v)$ and $\d_1^+X(v)  \subset \d_1^\star X$.

By an argument as in \cite{K1}, Corollary 3.3, in the vicinity of $\d_1^\star X$, we can deform the field $v$ to a new $f$-gradient-like field so that all the manifolds $\d_1^+X$, $\d_2^+X$,  \dots , $\d_{n - 1}^+X$, residing in  the component $\d_1^\star X$, will be diffeomorphic to balls, and $Z^1 =_{\mathsf{def}} \d_n^+X$ will consist of a number of arcs. The argument in Corollary 3.3 from \cite{K1} constructs such a $v$ to be boundary generic in the sense of Definition 2.1, \cite{K1}. Moreover, by \cite{K2}, Theorem 3.5, we can further perturb $v$ inside $X$,  without changing it on $\d_1X$, so that the new perturbation will be a traversally  generic field. Abusing notations, we continue to denote the new field  by $v$. 
\smallskip

The locus $\d_{n + 1 - k}^+X(v)$ is diffeomorphic to the disk $D^k$. We notice that, for $k > 2$, for each point $x \in D^k$, the $v$-trajectory $\g_x$ has at least one tangency point of  multiplicity $n + 1 - k$ residing in $\d_1^+X$, namely  $x$ itself. Thus, $\Gamma$ maps $D_k$ onto $\mathcal T_{\{max \geq n + 1- k\}}^+(v)$. Similarly,   $\Gamma : Z^1 \to  \mathcal T_{\{max \geq n\}}^+(v)$, $\Gamma: Z^0 \to \mathcal T_{\{max \geq n + 1\}}^+(v)$ are surjective maps.\smallskip

We notice that, due to the convexity of the flow in their neighborhoods, the points of $\d_j^-X$ are ``protected" in the following sense:  no $v$-trajectory can reach $\d_2^-X(v)$, unless the trajectory is a singleton which belongs to $\d_2^-X(v)$ in the first place, no $v$-trajectory  can reach $\d_3^-X(v)$, unless the trajectory is a singleton $\d_3^-X(v)$, and so on ... 
In particular, no $v$-trajectory through a point of $\d_1^+X(v)$ can reach $\d_2^-X(v)$, unless the trajectory is a singleton which belongs to $\d_2^-X(v)$ in the first place, no $v$-trajectory through a point of $\d_2^+X(v)$ can reach $\d_3^-X(v)$, unless the trajectory is a singleton $\d_3^-X(v)$, and so on ...  The claim also follows from Theorem 2.2, \cite{K2}.  

Therefore all the maps $\{\Gamma: \d_j^-X(v) \setminus \d_{j+1}X(v) \to \mathcal T(v)\}_j$ are $1$-to-$1$. Thus the claim in the third bullet has been validated.\smallskip

For a traversally generic $v$, by Corollary 5.1 from \cite{K3}, the map $\Gamma: \d_1X \to \mathcal T(v)$ is $(n + 2)$-to-$1$ at most. Since each trajectory, distinct from a singleton, must exit through $\d_1^-X$ at a point of an odd multiplicity, the same argument shows that $\Gamma: \d_1^+X \to \mathcal T(v)$ is $(n + 1)$-to-$1$ at most. Because for $v \in \mathcal V^\ddagger(X)$, the tangent spaces to $\d_j^+X^\circ$ along each trajectory $\g$ must form, with the help of the flow, a stable configuration in the germ of a $n$-section $S$, transversal to $\g$ (see \cite{K2}, Definition 3.2). 
Since $\dim(T_x(\d_j^+X)) = n +1 - j$ for every point $x \in \g \cap \d_j^+X^\circ$ and the flow-generated images of the spaces $\{T_x(\d_j^+X)\}_{x \in \g \cap \d_j^+X^\circ}$ must be in general position in a $n$-dimensional space $S$, the cardinality of the set $\g \cap \d_j^+X^\circ$ cannot exceed $\lceil \frac{n}{j -1} \rceil = \lceil \frac{n}{n - k} \rceil$, provided $k < n$. The statement in second bullet has been established. 
\smallskip

By the second bullet of Theorem 3.4, \cite{K2}, the smooth topological type of the stratification $\{\d_jX(v)\}_j$ is stable under perturbations of $v$ within the space $\mathcal B^\dagger(X)$ of boundary generic fields. The same argument shows that $\{\d_j^+X(v)\}_j$ is stable as well. Thus, for all fields $v'$ sufficiently close to $v$, the stratification $\{\d_j^+X(v')\}_j$ will remain as in (\ref{eq6.47}). By Theorem 3.5 from \cite{K2}, all vector fields, sufficiently close to a traversally  generic vector field, will remain traversally generic. Therefore this fact gives the desired control of the cardinality for the fibers of the maps $\Gamma: \d_j^+X(v') \to \mathcal T(v')$ and of the smooth topology of the stratification $\{\d_j^+X(v')\}_j$ within an open neighborhood of $v$ in $\mathcal V^\ddagger(X)$.  
\hfill 
\end{proof}

{\bf Remark 3.1.}
\noindent Recall that the trajectory space $\mathcal T(v)$ in the Origami Theorem \ref{th6.11}  is not only weakly homotopy equivalent to the manifold $X$ (see \cite{K3}, Theorem 5.1), but also carries a $n$-bundle $\tau$ whose pull-back under $\Gamma$ is stably isomorphic to the tangent bundle $TX$ (\cite{K4}, Lemma 2.1). As a result, $\tau$ and $TX$ share all stable  characteristic classes. So all this information about $X$ is hidden in a subtle way in \emph{the geometry of the origami map} $\Gamma: D^n \to \mathcal T(v)$. \hfill $\diamondsuit$ 
\smallskip

The Origami Theorem \ref{th6.11} oddly resembles the Noether Normalization Lemma in the Commutative Algebra \cite{No}, however, with the direction of the ramified morphism being reversed. Recall that, in its algebro-geometrical formulation, the Normalization Lemma states that any affine variety is a branched covering over an affine space. In  contrast, in our setting, many trajectory spaces $\mathcal T(v)$---rather intricate objects---have a simple and \emph{universal} ramified cover---the ball. 
\smallskip

To explain this analogy, for a traversally generic field $v$, consider the Lie derivation $\mathcal L_v$ of the algebra $C^\infty(X)$ of smooth functions on $X$. Its kernel $C^\infty(\mathcal T(v))$\footnote{$C^\infty(\mathcal T(v))$ is just a subalgebra of $C^\infty(X)$, not an ideal.} is a part of the long exact sequence of vector spaces: 
 $$0 \to C^\infty(\mathcal T(v)) \rightarrow C^\infty(X) \stackrel{\mathcal L_v}{\longrightarrow} C^\infty(X) \rightarrow \dots$$  By Definition \ref{def3.1}, $C^\infty(\mathcal T(v))$, the algebra of smooth functions on the space of trajectories, can be identified with the algebra of all smooth functions on $X$ that are constant along each $v$-trajectory.\smallskip

When a traversally  generic $v$ is such that $\d_1^+X(v)$ is diffeomorphic to $D^n$, then  employing Theorem \ref{th6.11} and with the help of the finitely ramified surjective map $$\Gamma_\d: D^n = \d_1^+X(v) \subset X \stackrel{\Gamma}{\longrightarrow} \mathcal T(v),$$ we get the induced monomorphism $\Gamma_\d^\ast: C^\infty(\mathcal T(v)) \to C^\infty(D^n)$ of algebras, where the target algebra $C^\infty(D^n)$ of smooth functions on the $n$-ball is \emph{universal} for a given dimension $n$. \smallskip

Any point-trajectory $\g \in \mathcal T(v)$ gives rise to the maximal ideal $\mathsf m_\g\, \lhd \, C^\infty(\mathcal T(v))$, comprising smooth functions on $\mathcal T(v)$ that vanish at $\g$. On the other hand, if $\mathsf m \, \lhd \, C^\infty(\mathcal T(v))$ is a maximal ideal and a function $h \in \mathsf m$ does not vanish on the compact $\mathcal T(v)$, then the function $1 = (\frac{1}{h})\cdot h \in \mathsf m$, so that $\mathsf m = C^\infty(\mathcal T(v))$. Thus every maximal ideal $\mathsf m \lhd C^\infty(\mathcal T(v))$, distinct from the algebra itself, is of the form $\mathsf m_\g$. \smallskip

The map $\Gamma^\d$ is finitely ramified with fibers of cardinality $(n+1)$ at most (\cite{K3}, Corollary 5.1). Therefore,  for any maximal ideal $\mathsf m \lhd  C^\infty(\mathcal T(v))$, its $\Gamma^\d$-induced image $(\Gamma^\d)^\ast(\mathsf m)$ is the intersection $\bigcap_i \,\mathsf m_i$ of $(n+1)$ maximal ideals $\mathsf m_i \lhd C^\infty(D^n)$ at most.  

One can think of smooth vector fields on $X$ as derivatives of the algebra $C^\infty(X)$. We denote the space of such operators by the symbol $\mathcal D(X)$. Let $\mathcal C^+(X) \subset C^\infty(X)$ denote the open cone, formed by all strictly positive functions. The gradient-like fields $v$ correspond to derivatives $\mathcal L_v \in \mathcal D(X)$ such that $\mathcal L_v(f) \in \mathcal C^+(X)$ for some $f \in  C^\infty(X)$.\smallskip

By Theorem 3.5 from \cite{K2}, the traversally  generic fields form a nonempty open set $\mathcal V^\ddagger(X)$ in the space of all vector fields and an open and dense set in the space of all traversing vector fields. Therefore, the previous considerations lead to the following reformulation of Theorem \ref{th6.11}.

\begin{corollary}\label{cor8.5}{\bf (The origami resolutions $\mathbf {C^\infty(D^n)}$ for the kernels of special derivatives of the algebra $\mathbf {C^\infty(X)}$)} \smallskip

Let $C^\infty(D^n)$ denote the algebra of smooth functions on the $n$-ball.

For any $(n+1)$-dimensional smooth connected and compact manifold $X$ with boundary, there exists an open nonempty subset $\mathcal Der^\odot(X) \subset \mathcal Der(X)$ of algebra derivatives $\mathcal L: C^\infty(X) \to C^\infty(X)$ that possess the following properties:  
\begin{itemize}
\item for any $\mathcal L \in \mathcal Der^\odot(X)$, there exists a function $f \in C^\infty(X)$ such that $\mathcal L(f) \in \mathcal C^+(X)$, the positive cone, 

\item for any $\mathcal L \in \mathcal Der^\odot(X)$, there exists a monomorphism of algebras\footnote{that is induced by the origami map $\Gamma: D^n \to \mathcal T(v)$} 
$$(\Gamma^\d)^\ast: \ker(\mathcal L) \to C^\infty(D^n)$$ such that, for any maximal ideal $\mathsf m \, \lhd \, \ker(\mathcal L)$, the image $(\Gamma^\d)^\ast(\mathsf m)$ is an intersection $\bigcap_i\, \mathsf m_i$ of $n +1$ maximal ideals $\mathsf m_i \lhd C^\infty(D^n)$ at most. \hfill $\diamondsuit$
 \end{itemize}
\end{corollary}
\smallskip


We would like to learn the answer to the following question: 

\begin{question}\label{q6.3} Let $X$ be a compact connected smooth manifold with boundary. Let $v$ be a traversing and boundary generic vector field on $X$. Describe the image under the restriction map, induced by the inclusion $\d X \subset X$, of the algebra $\mathsf{ker}(\mathcal L_v)$ in the algebra $C^\infty(\d X)$. \hfill $\diamondsuit$.
\end{question}

\section{A glimpse of holography}

We devote this section to the fundamental phenomenon of the \emph{holography} of traversing flows. Crudely, we are concerned with the ability to reconstruct the manifold $X$ and the traversing flow $v$ (rather, the 1-dimensional oriented foliation $\mathcal F(v)$, generated by $v$) on it in terms of some data, generated by the flow on the boundary $\d X$. This kind of problem is in the focus of an active research in Differential Geometry, where it is known under the name of \emph{geodesic inverse scattering problem} \cite{BCG}, \cite{Cr}, \cite{Cr1}, \cite{CEK}, \cite{SU}-\cite{SU2}, \cite{SUV}, \cite{SUV1}.

The main result of this section, Theorem \ref{th4.1}, describes some boundary data, sufficient for a reconstruction of the pair $(X, \mathcal F(v))$, up to a homeomorphism. The reader interested in further developments of these ideas may glance at the paper \cite{K4},  \cite{K5}, and the forthcoming book \cite{K6}.
\smallskip

First, we introduce one basic construction (see Figure 2) which will be very useful throughout our investigations.

\begin{figure}[ht]\label{fig2}
\centerline{\includegraphics[height=1.8in,width=2.9in]{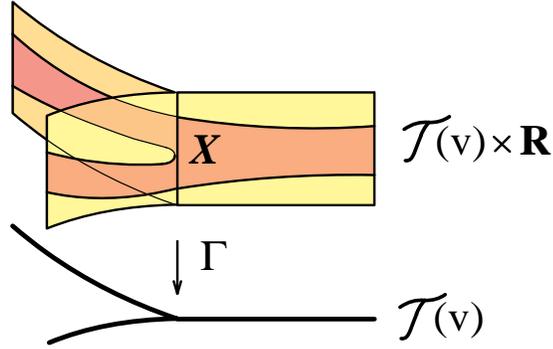}}
\bigskip
\caption{\small{The embedding $\a(f, v)$ of $X$ into the product $\mathcal T(v) \times \R$.}} 
\end{figure}

\begin{definition}\label{def3.1}
We say that a function $h: \mathcal T(v) \to \R$ is \emph{smooth}, if its pull-back $\Gamma^\ast h: X \to \R$, under the obvious map $\Gamma: X \to \mathcal T(v)$, is smooth. \hfill  $\diamondsuit$
\end{definition}

\begin{lemma}\label{lem6.17} Consider a traversing vector field $v$ on a compact smooth connected manifold $X$ with boundary and a Lyapunov function $f: X \to \R$, $df(v) > 0$. \smallskip

Any such pair $(v, f)$ generates an embedding $\a(v, f): X \subset \mathcal T(v) \times \R$, where $\mathcal T(v)$ denotes the trajectory space.\smallskip

For any smooth map $\b: \mathcal T(v) \to \R^N$, the composite map $$A(v, f): X \stackrel{\mathcal \a}{\longrightarrow} \mathcal T(v) \times \R \; \stackrel{\b \times id}{\longrightarrow}\; R^N \times \R$$ is smooth.
\smallskip

Any two embeddings, $\a(f_1, v)$ and $\a(f_2, v)$, are isotopic through homeomorphisms, provided that $df_1(v) > 0, df_2(v) > 0$.
\end{lemma}

\begin{proof} Since $f$ is strictly increasing along the $v$-trajectories, any point $x \in X$ is determined by the $v$-trajectory $\g_x$ through $x$ and the value $f(x)$.  Therefore, $x$ is determined by  the point $\g_x \times f(x) \in \mathcal T(v) \times \R$. By the definition of topology in $\mathcal T(v)$, the correspondence $\a(f, v): x \rightarrow \g_x \times f(x)$ is a continuous map.\smallskip

In fact, $\a(f, v)$ is a smooth map in the spirit of Definition \ref{def3.1}: more accurately,  for any map $\b: \mathcal T(v) \to \R^N$, given by $N$ smooth functions on $\mathcal T(v)$, the composite map $A(v, f): X \to \R^N \times \R$ is smooth.  The verification of this fact is on the level of definitions. 
\smallskip

For a fixed $v$, the condition $df(v) > 0$ defines an open convex cone $\mathcal C^+(v)$ in the space $C^\infty(X)$. Thus, $f_1$ and $f_2$ can be linked by a path in $\mathcal C^+(v)$, which results in  $\a(f_1, v)$ and $\a(f_2, v)$ being homotopic through homeomorphisms. 
\hfill 
\end{proof}

\noindent {\bf Remark 3.1.}
 By examining  Figure 2, we observe an interesting phenomenon: the embedding $\a: X \subset \mathcal T(v) \times \R$ does not extend to an embedding of a larger manifold $\hat X \supset X$, where $\hat X \setminus X \approx \d_1X \times [0, \e)$. In other words, $\a(\d_1X)$ has no outward ``normal field" in the ambient $\mathcal T(v) \times \R$. In that sense, $\a(\d_1X)$ is \emph{rigid} in $\mathcal T(v) \times \R$! \hfill $\diamondsuit$ 
\smallskip

\begin{corollary}\label{cor3.2} Any compact connected smooth $(n+1)$-manifold $X$ with boundary admits an embedding $\a: X \to \mathcal T \times \R$, were $\mathcal T$ is a $CW$-complex that is the image of the $n$-ball $D^n$ under a continuous map, whose fibers are of the cardinality $n+1$  at most. Moreover, $\a$ is a  homotopy equivalence.\end{corollary}

\begin{proof} We combine the Origami Theorem \ref{th6.11} with Lemma \ref{lem6.17} to validate the first claim of the lemma. 

Since $\Gamma = p \circ \a$, where $p: \mathcal T \times \R \to \mathcal T$ is the obvious projection, and $\Gamma$ is a homotopy equivalence by Theorem 5.1 from \cite{K3}, so is the map $\a$.
\hfill
\end{proof}

\begin{corollary}\label{cor4.2} Let $v$ be a traversing vector field on a compact smooth connected manifold $X$ with boundary and $f: X \to \R$ its Lyapunov function. Let $X^\circ$ denote the interior of $X$. \smallskip 

Then the embedding $$\a(f, v):\, \d_1X \longrightarrow \big(\mathcal T(v) \times [0, 1]\big) \setminus \a(f, v)(X^\circ)$$ is a homology equivalence. As a result, the space $$\big(\mathcal T(v) \times [0, 1]\big) \setminus \a(f, v)(X^\circ)$$ is a Poincar\'{e} complex of the formal dimension $\dim(X) - 1$.
\end{corollary}

\begin{proof} Put $\a =_{\mathsf{def}} \a(f, v)$. Let us compare the homology long exact sequences of the two pairs: $$X \supset \d_1X\;\; \text{and} \;\; \mathcal T(v) \times [0, 1] \supset \, (\mathcal T(v) \times [0, 1]) \setminus \a(X^\circ).$$ They are connected by the vertical homomorphisms that are induced by $\a$. Using the excision property, $$\a_\ast: H_\ast(X, \d_1X) \to H_\ast\big(\mathcal T(v) \times [0, 1], \; (\mathcal T(v) \times [0, 1]) \setminus \a(X^\circ)\big)$$ are isomorphisms. On the other hand, since by Theorem \ref{th6.8},  $\Gamma: X \to \mathcal T(v)$ is a homology equivalence, $\a_\ast: H_\ast(X) \to H_\ast(\mathcal T(v) \times [0, 1])$ are isomorphisms. Therefore by the Five Lemma, $$\a_\ast: H_\ast(\d_1X) \to H_\ast((\mathcal T(v) \times [0, 1]) \setminus \a(X^\circ))$$ must be isomorphisms as well. Since $\d_1X$ is a closed $n$-manifold, it is a Poincar\'{e} complex of formal dimension $n$, and thus so is the space $$(\mathcal T(v) \times [0, 1]) \setminus \a(f, v)(X^\circ).$$
\hfill 
\end{proof}


We denote by $\mathcal F(v)$ the oriented 1-dimensional foliation on $X$, produced by the  $\hat v$-trajectories. 

\begin{definition}\label{def4.2} Let $v$ be a traversing vector field on $X$. Given two points $x, y \in \d_1X$, we write $y \succ_v x$ if both points belong to the same $v$-trajectory $\g \subset X$ and, moving from $x$ in the $v$-direction along $\g$, we can reach $y$. 

The relation $y \succ_v x$ introduces a \emph{partial order} $\succ_v$ in the set $\d_1X$. 
\hfill $\diamondsuit$
\end{definition}

Adding an extra ingredient to the partial order $\succ_v$ (equivalently, to the origami construction), allows for a reconstruction of the topological type of the pair $(X, \mathcal F(v))$ from the flow-generated information, \emph{residing on the boundary} $\d X$. The new ingredient is the restriction of the Lyapunov function $f: X \to \R$ to the boundary.

\begin{theorem}\label{th4.1}{\bf (Topological Holography of Traversing Flows)}
\smallskip 

\noindent Let a $v$ be a traversing vector field $v$ on a compact connected smooth manifold $X$ with boundary, and let $f: X \to \R$ be its Lyapunov function. \smallskip

Then the partial order $\succ_v$ on $\d_1X$, 
together with the restriction $f^\d: \d X \to \R$ of a Lyapunov function $f: X \to \R$, allows for a reconstruction of the topological type of the pair $(X, \mathcal F(v))$.
 \end{theorem}
 
 \begin{proof} The validation of the theorem is based on Lemma \ref{lem6.17}.
 \smallskip
 
First, we observe that the partial order $\succ_v$ allows for a reconstruction of the trajectory space $\mathcal T(v)$ and the quotient map $\Gamma^\d: \d_1X \to \mathcal T(v)$. Indeed, we declare two points $x, y \in \d_1X$ \emph{equivalent} if $y \succ_v x$ or $x \succ_v y$. This equivalence relation $\sim_v$  produces the quotient map $\Gamma^\d: \d_1X \to (\d_1X)\big /\sim_v$, whose target may be identified with the space  $\mathcal T(v)$ since, for a traversing $v$, every trajectory $\g$ is determined by its intersection $\g\cap \d_1X$.   \smallskip 
 
As in Lemma \ref{lem6.17}, using $f^\d$, we construct an embedding $$\a^\d = \a(v, f^\d)|: \d X \subset \mathcal T(v) \times \R.$$ Then $\a^\d(\d X)$ divides $\mathcal T(v) \times \R$ into two domains, one of which is compact. That compact domain $\mathcal X$ is $\a(v, f)(X)$. Since $\a(v, f): X \to \mathcal X$ is a homeomorphism, we managed to reconstruct the topological type of $X$ from the boundary data $(\succ_v, f^\d)$ (in the end, from $(\Gamma^\d, f^\d)$). \smallskip

Evidently, $\mathcal X$ is equipped with a 1-dimensional foliation $\mathcal G$, generated by the product structure in the ambient $\mathcal T(v) \times \R$. 

By its construction, the homeomorphism $\a(v, f)$ maps each leaf of $\mathcal F(v)$ to a leaf of $\mathcal G$. Thanks to $\a(v, f)$, the pair $(\mathcal X, \mathcal G)$, which we have recovered from the boundary data $(\succ_v, f^\d)$, has the same topological type as the original pair $(X, \mathcal F(v))$. 

Note that, for a given pair $(\succ_v, f^\d)$, the homeomorphism $\a(v, f)$ is far from being unique. Even, for a fixed pair $(X, v)$, we may vary the Lyapunov function $f$, while keeping $f^\d$ fixed. However, the space $\mathsf{Lyap}(v, f^\d)$ of such Lyapunov functions $f$ is convex, and thus contractible.  Therefore, for any two $f_1, f_2 \in \mathsf{Lyap}(v, f^\d)$,\, the embeddings $\a(v, f_1)$ and $\a(v, f_2)$ are homotopic through homeomorphisms that map $\mathcal F(v)$ to $\mathcal G$.
\hfill
 \end{proof}

\begin{remark}\label{rem4.1} The question whether the data $(\succ_v, f^\d)$ are sufficient for a reconstruction of the \emph{differentiable} or even \emph{smooth} topological type of the pair $(X, \mathcal F(v))$ seems to be much more delicate. We suspect that the positive answer to it will depend on our ability to answer Question \ref{q6.3}. \smallskip

Under certain assumptions about $v$ (such as some restrictions on the combinatorial types of $v$-trajectories), the answer is positive \cite{K4}. 
\hfill $\diamondsuit$
\end{remark}

\begin{corollary}\label{cor4.3} Let a traversally generic vector field $v$ on $X$ be such that $\d_1^+X(v) \approx D^n$ \footnote{by Theorem \ref{th4.1}, such vector field exists.}. Then the origami map $\Gamma^\d: D^n \to \mathcal T(v)$, together with the restriction $f^\d_+: D^n \to \R$ of the Lyapunov function $f: X \to \R$, allow for a reconstruction of the topological type of the pair $(X, \mathcal F(v))$.\smallskip 
\end{corollary}
 
 \begin{proof}  To validate the claim of the theorem, we combine Theorems \ref{th6.11} and \ref{th4.1}. \smallskip
 
Given a traversing $v$ and a function $f^\d_+: \d_1^+X(v) \to \R$, let $\mathsf{Lyap}(v, f^\d_+)$ be the space of Lyapunov functions $f: X \to \R$ such that $f|_{\d_1^+X(v)}  = f^\d_+$. Again, $\mathsf{Lyap}(v, f^\d_+)$ is a convex contractible space. \smallskip
 
We assume that the function $f^\d_+$ is known and is generated by \emph{some} (unknown) $f \in \mathsf{Lyap}(v)$. By the properties of the Lyapunov function $f$, we may assume that $f^\d_+$ extends to a function $f^\d: \d_1X \to \R$ so that, for any $x, y \in \d_1X$, $y \succ_v x$, the inequality $f(x) < f(y)$ is valid. \smallskip  
 
By Theorem \ref{th6.11}, the image $\Gamma^\d(D^n)$ of the Origami map $\Gamma^\d$ is the trajectory space $\mathcal T(v)$, and the fibers of $\Gamma^\d$ may be identified with the $(\sim_v)$-equivalence classes of points in $D^n$. Moreover, since $v$ is traversally generic, by Theorem 5.1 from \cite{K3}, $\mathcal T(v)$ is a compact $CW$-complex.
\smallskip
 
Now, as in Lemma \ref{lem6.17}, using $f^\d$, we construct an embedding $\a(v, f^\d): \d X \subset \mathcal T(v) \times \R$. The rest of the argument is similar to the argument we used to prove Theorem \ref{th4.1}.
\hfill
\end{proof}


\end{document}